\documentclass{amsart}

\usepackage{amsmath}
\usepackage{amsfonts}
\usepackage{amssymb}
\usepackage{amsthm}
\usepackage{mathrsfs}

\newtheorem{theorem}{Theorem}[section]

\newtheorem{lemma}[theorem]{Lemma}
\newtheorem{proposition}[theorem]{Proposition}
\newtheorem{corollary}[theorem]{Corollary}

\newtheorem*{remark}{Remark}

\newtheorem{question}{Question}

\newtheorem{definition}[theorem]{Definition}
\newtheorem*{maintheorem}{Theorem \ref{maintheorem}}

\newcommand{\cof}[1]{\mathrm{cof}(#1)}
\newcommand{\cin}[1]{\mathrm{cin}(#1)}
\newcommand{\cov}[1]{\mathrm{cov}(#1)}

\newcommand{\baire}{\omega^\omega}
\newcommand{\bairetree}{\omega^{<\omega}}
\newcommand{\Baire}{[\omega]^\omega}
\newcommand{\cone}[1]{\langle #1 \rangle}
\newcommand{\ultpow}[1]{#1 / \mathscr{U}}
\newcommand{\vltpow}[1]{#1 / \mathscr{V}}
\newcommand{\ult}[1]{\mathscr{#1}}

\title{Michael spaces and Ultrafilters}
\author{Arturo Mart\'{i}nez-Celis}
\address{Instytut Matematyczny, Uniwersytet Wrocławski, pl. Grunwaldzki 2, 50-384 Wrocław, Poland}
\email{arturo.martinez-celis@math.uni.wroc.pl}


\begin{document}
\begin{abstract}
    A Michael space is a Lindel\"of space which has a non-Lindel\"of product with the Baire space. In this work, we present the notion of Michael ultrafilter and we use it to construct a Michael space under the existence of a selective ultrafilter and $\max \{ \mathfrak{b}, \mathfrak{g} \} =\mathfrak{d}$.
\end{abstract}

\maketitle

In 1963, E. Michael \cite{MichaelNormal} constructed, under CH ($\mathfrak{b}=\aleph_1$) a \emph{Michael space}, a Lindel\"of space whose product with the Baire space is not Lindel\"of. Later, in 1990, K. Alster \cite{AlsterMichaelspace} constructed one under MA. In 1999, J. Moore \cite{MooreCombinatorics} developed a general framework for constructing Michael spaces and was able to construct one under $\mathfrak{d}=\cov{\mathcal{M}}$. This framework can be stated using the notion of $\theta$-Michael sequences:

\begin{definition}{\cite{MooreCombinatorics}}
    A sequence $\{ X_\alpha : \alpha \leq \theta \}$ is a $\theta$\emph{-Michael sequence} if
    \begin{itemize}
        \item it is a strictly $\subseteq$-increasing sequence of subsets of $\baire$, with $X_\theta = \baire$,
        \item for every compact set $K \subseteq \baire$ and every $\alpha$ of uncountable cofinality, if $K \subseteq X_\alpha$ then there is a $\gamma < \alpha$ such that $K \subseteq X_\gamma$.
    \end{itemize}
\end{definition}

For our convenience, we slightly modified the original notion of Michael sequences (our notion can be obtained by taking the sequences of complements of the original notion). In \cite{MooreCombinatorics}, it is proven that the existence of a $\theta$-Michael sequence with $\theta$ of uncountable cofinality implies the existence of a Michael space. So we will focus on construct a Michael sequence instead of directly constructing a Michael space.

In this work we will use frequently the order structure given by an ultrafilter: Given a non-principal ultrafilter $\ult{U}$ over $\omega$ and $f,g \in \baire$, $f$ is $\leq_\ult{U}$\emph{-dominated by} $g$, denoted by $f \leq_\ult{U} g$, if $\{ n \in \omega : f(n) \leq g(n) \} \in \ult{U}.$ The \emph{character of the ultrafilter}, denoted by $\chi(\ult{U})$, is the smallest size of a basis for $\ult{U}$, that is the smallest $B \subseteq \ult{U}$ such that for every $U \in \ult{U}$ there is $V \in B$ such that $V \subseteq U$. The cardinal $\mathfrak{u}$, \emph{the ultrafilter number}, is the smallest $\chi(\ult{U})$. If $K \subseteq \baire$, we will denote by $\cof{\ultpow{K}}$ as the smallest $|D|$ such that $D \subseteq K$ and for every $f \in K$, there is $g \in D$ such that $f \leq_\ult{U} g$. The sets $D$ that satisfy the last property will be called an $\leq_\ult{U}$\emph{-cofinal set} or a $\leq_\ult{U}$\emph{-dominating set in} $K$. The cardinal $\cof{\ultpow{\baire}}$ is known in the literature as \emph{the cofinality of the ultrapower}, and has been studied in \cite{Canjar2}, \cite{Canjar1}, \cite{Nyikosultrafilter} and \cite{BlassHeike}. These can be easily shown to be regular. We will also use \emph{the coinitiality of the ultrapower}, denoted by $\cin{\ultpow{\baire}}$, which is the smallest collection $D$ of finite to one non-decreasing functions of $\baire$ such that for any non decreasing finite to one $f \in \baire$ there is $g \in D$ such that $g \leq_\ult{U} f$. Such cardinal was considered by Canjar in \cite{Canjar2}. An ultrafilter $\ult{U}$ over $\omega$ is a \emph{p-point} if every decreasing sequence $\{ U_i : i \in \omega \} \subseteq \ult{U}$ has a \emph{pseudointersection} in $\ult{U}$, that is, a set $U \in \ult{U}$ such that $U \subseteq^*U_n$ for all $n \in \omega$. An ultrafilter $\ult{U}$ is a \emph{q-point} if for every interval partition $\langle I_n : n \in \omega \rangle$ of $\omega$ there is $U \in \ult{U}$ such that $|U \cap I_n| \leq 1$. A \emph{selective ultrafilter} is an ultrafilter that is a p-point and a q-point at the same time.

We will also use the theory of cardinal invariants. In particular, we will be using the following well-known cardinal invariants of the continuum:
\begin{itemize}
    \item The \emph{unbounding number}, denoted by $\mathfrak{b}$, is the smallest cardinality of a set $B \subseteq \baire$ such that for every $g \in \baire$ there is $f \in B$ such that $f \nleq^* g$,
    \item the \emph{dominating number}, denoted by $\mathfrak{d}$, is the smallest cardinality of a set $D \subseteq \baire$ such that for every $g \in \baire$ there is $f \in D$ such that $g \leq^* f$,
    \item the \emph{covering of the meager ideal}, denoted by $\cov{\mathcal{M}}$ is the smallest amount of meager sets required to cover the real line.
    \item the \emph{groupwise dense number}, denoted by $\mathfrak{g}$, is the smallest cardinality of a collection of group-wise dense sets with empty intersection; where a set $G \subseteq [\omega]^\omega$ is \emph{group-wise dense} if $G$ is closed under subsets and finite modifications, and for every partition $\langle I_n : n\in \omega \rangle$ of $\omega$ there is an infinite $A \subseteq \omega$ such that $\bigcup_{n\in A} I_n \in G$.    
\end{itemize}

One can show that $\max \{ \mathfrak{b},\mathfrak{g},\cov{\mathcal{M}} \} \leq \mathfrak{d}$ and that $ \max \{ \cov{\mathcal{M}},\mathfrak{b}\} \leq \mathfrak{u}$ (see for example \cite{BlassSurvey}). It is also true that $\mathfrak{g}$ is smaller or equal than the successor of $\mathfrak{b}$ (see \cite{Shelahgandb}). Any other inequality among any pair of these cardinal invariants is consistent (see \cite{Barty} and \cite{BlassSurvey}). It is also true that $\max \{ \mathfrak{b} , \mathfrak{g}\} \leq \cof{\ultpow{\baire}} \leq \mathfrak{d}$ (see \cite{BlassHeike}). An analogous result about the coinitiality is also true (Theorem \ref{blassheikeforcoinitiality}).

The main theorem of this work is the following.
\begin{maintheorem}
    A selective ultrafilter $\ult{U}$ is Michael if and only if $\cof{\ultpow{\baire}} \leq \cin{\ultpow{\baire}}$.
\end{maintheorem}

In particular, by Theorem \ref{existanceofmichaelspace}, there is a Michael space whenever there is a selective ultrafilter and either $\mathfrak{b} = \mathfrak{d}$ or $\mathfrak{g} = \mathfrak{d}$, for example, after forcing with $\mathcal{P}(\omega) / \mathrm{Fin}$.

The paper is organized in three sections: In the first section we will introduce the concept of Michael ultrafilters, we will look at some basic properties relating compact sets and ultrafilters and we will prove that $\cov{\mathcal{M}} = \mathfrak{c}$ imply the existence of this kind of ultrafilters. In the second section, we will focus mostly into the relation between q-points, selective ultrafilters and compact sets, to finally be able to prove Theorem \ref{maintheorem}. In the final section, we will look into the relationship of the Rudin-Keisler order, the Rudin-Blass order and Michael ultrafilters, to finally conclude with a model where no Michael ultrafilters can exist.

Our notation is standard and mostly follows \cite{Barty}: If $f,g \in \baire$, $A,B \subseteq \omega$ then $f \leq^* g$ means that $\{ n \in \omega : f(n) > g(n) \}$ is finite, $A \subseteq^* B$ means that $B \setminus A$ is finite, $f \leq_A g$ means that $A \subseteq \{ n \in \omega : f(n) \leq g(n) \}$ and $f | A$ is the usual restriction to $A$, that is, a function in $\baire$ such that $f | A (n) = f(n)$ if $n\in A$, otherwise the value is $0$. If $s \in \bairetree$ then $\langle s \rangle = \{ f \in \baire : s \subseteq f \}$, if $T \subseteq \bairetree$ is a tree, then $[T] = \{ f \in \baire : \forall n \in \omega (f | n \in T) \}$.  The filters considered in this manuscript are assumed to be non-principal. If $\mathscr{F}$ is a filter basis on $\omega$, then $\mathscr{F}^+ = \{ A\subseteq \omega : \forall F \in \mathscr{F} (A\cap F \text{ is infinite }) \}$. The sets $\mathcal{P}(\omega)$ and $\Baire$ are the collections of all sets of $\omega$, and the collection of all infinite subsets of $\omega$, respectively. In this manuscript, every enumeration for a subset of $\omega$ will be assumed to be increasing.

\section{Michael Ultrafilters}

Given any $\leq_\ult{U}$-dominant family, one can attempt to construct a Michael sequence in the same way as J. Moore did in \cite{MooreCombinatorics}, although there are some situations where this will fail, as we will see in the last section of this manuscript. One of the advantages of considering the order $\leq_\ult{U}$ is that it is linear, so unbounded sets and dominating sets coincide, and we will implicitly use this to isolate the properties that we require so that any $\leq_\ult{U}$-dominating family yields a Michael sequence. We begin this section by introducing the notion of ultrafilters that we will be working with throughout this paper.

\begin{definition}
An ultrafilter $\mathscr{U}$ is a Michael ultrafilter if for every compact $K \subseteq \baire$, if $\cof{\ultpow{K}} > \omega$, then $\cof{\ultpow{K}}\geq \cof{\ultpow{\baire}}$.
\end{definition}

 Michael ultrafilters are the ultrafilters that generate a Michael sequence, no matter which dominant family of minimal size we work with.

\begin{theorem}
If there is a Michael Ultrafilter, then there is a Michael space. \label{existanceofmichaelspace}
\end{theorem}
\begin{proof}
    Let $\mathscr{U}$ be a Michael ultrafilter, we will construct a $\cof{\ultpow{\baire}}$-Michael sequence: Let $\{ f_\alpha : \alpha \in \cof{\ultpow{\baire}} \}$ be a $\leq_\mathscr{U}$-dominating family. For $\alpha \in \cof{\ultpow{\baire}}$, let $X_\alpha = \{ f \in \baire : \exists \beta < \alpha \ (f \leq_\mathscr{U} f_\beta) \}$. We will now show that $\{ X_\alpha : \alpha \in \cof{\ultpow{\baire}} \}$ is a Michael sequence. Suppose there is a compact set $K \subseteq \baire$ and an $\alpha$ of uncountable cofinality such that $K \subseteq X_\alpha$. For each $\beta < \alpha$ pick (if possible) a function $g_\beta \in K$ such that $f_\beta \leq_\mathscr{U} g_\beta$. Then, it follows that $\{ g_\beta : \beta < \alpha \}$ is a $\leq_\mathscr{U}$-dominating family inside $K$. Since $\mathscr{U}$ is a Michael ultrafilter, then $\cof{\ultpow{K}} \leq \omega$. Let $\{ g_n : n \in \omega \}$ be a $\leq_\mathscr{U}$-dominating family in $K$. For every $n \in \omega$, pick $\beta_n \in \alpha$ such that $g_n \leq_\mathscr{U} f_{\beta_n}$. Then it follows that $K \subseteq X_\gamma$, where $\gamma = \sup \{ \beta_n : n \in \omega \}$.
\end{proof}

Consistently one can easily build these objects. For example, if $\ult{U}$ is an ultrafilter such that $\cof{\ultpow{\baire}} = \aleph_1$, then $\ult{U}$ has to be Michael and since $\cof{\ultpow{\baire}} \leq \mathfrak{d}$ then all ultrafilters are Michael under the assumption $\mathfrak{d} = \aleph_1$. One of the main goals of this section is to show that these ultrafilters can be constructed under $\cov{\mathcal{M}} = \mathfrak{c}$. Before that, we will require the following notion.

\begin{definition}
    Given a filter $\mathscr{F}$ on $\omega,$ a compact set $K = [T] \subseteq \baire$ is \emph{internally unbounded in }$\mathscr{F}$ if for every $f \in K$ and every $s \in T$ there is $g \in K$ extending $s$ such that $\{ n \in \omega : g(n) \geq f(n)\} \in \mathscr{F}$.
\end{definition}

 Observe that, if $K$ is a compact set internally unbounded in a filter $\mathscr{F}$, then, for every $F \in \mathscr{F}$, the compact set $K|F = \{ f|F : f \in K \}$ is also internally unbounded in $\mathscr{F}$. Internally unbounded compact sets are the ones that do not have bounded neighbourhoods according to $\ult{F}$; In fact, if we consider the usual order $\leq^*$ in internally unbounded compact sets, every function only bounds a meager part of it.

\begin{remark}
    If $\mathscr{F}$ is a filter, $f \in K$, $F \in \mathscr{F}^+$ and $K \subseteq \baire$ is internally unbounded in $\mathscr{F}$ or if $K = \baire$, then $\{ g|F  \in K|F : g|F \leq^* f|F \}$ is meager in $K$.
\end{remark}
\begin{proof}
    The sets $A_n = \{g \in K : \forall i \in F \setminus n \ (g(i) \leq f(i))\}$ are clearly nowhere dense.
\end{proof}

It turns out that every non-trivial compact set has an internally unbounded compact set inside it of the same $\leq_\ult{U}$-cofinality.

\begin{lemma}
Given an ultrafilter $\mathscr{U}$ and a compact $K \subseteq \baire$ such that $\cof{\ultpow{K}}>\omega$, there is a compact $K' \subseteq K$ internally unbounded in $\mathscr{U}$ such that $\cof{\ultpow{K}} = \cof{\ultpow{K'}}$.
\end{lemma}
\begin{proof}
Let 
$$
\Omega = \{ s \in \bairetree : \cone{s}\cap K \neq \emptyset \wedge \exists f_s \in K (\forall g \in \cone{s}\cap K (g \leq_{\mathscr{U}} f_s)) \}
$$
Define $K'= \bigcap_{s \in \Omega} K \setminus \cone{s}$. Note that $K'\subseteq K$ is a compact set such that $\cof{\ultpow{K}} = \cof{\ultpow{K'}}$ (as we only removed functions that were bounded by the $f_s$). We will now show that $K'$ is internally unbounded in $\mathscr{U}$. Let $t \in \bairetree$ be such that $\cone{t}\cap K' \neq \emptyset$ and let $f \in K'$. Pick $\hat{f} \in K$ such that $\hat{f} \geq_\mathscr{U} f,f_s$. Then, since $t \notin \Omega$, there is $g \in \cone{t}\cap K$ such that $g \geq_\mathscr{U} \hat{f}$. Then $g \geq_\mathscr{U} f$ since $g \geq_\mathscr{U} \hat{f}$ and $g \in K'$ since $g \geq_\mathscr{U} f_s$. Therefore $K'$ is internally unbounded in $\ult{U}$.
\end{proof}

In \cite{Barty} and \cite{CanjarGenericExistence}, the authors proved that $\cov{\mathcal{M}} = \mathfrak{c}$ is equivalent to the statement that any filter of character of less than $\mathfrak{c}$ can be extended to a selective ultrafilter. We will prove a similar result for Michael ultrafilters, giving a plethora of examples under $\cov{\mathcal{M}} = \mathfrak{c}$.

\begin{theorem}
    Under $\cov{\mathcal{M}}=\mathfrak{c}$, every filter of character smaller than $\mathfrak{c}$ can be extended to a Michael ultrafilter. \label{michaelcovmeager}
\end{theorem}
\begin{proof}
    Let $F_0$ be a filter basis of size smaller than $\mathfrak{c}$. Let $\{ A_\alpha : \alpha \in \mathfrak{c}\}$ be an enumeration of $\Baire$, $\{ f_\alpha : \alpha \in \mathfrak{c}\}$ an enumeration of $\baire$ and $\{ K_\alpha : \alpha \in \mathfrak{c}\}$ an enumeration of all compact sets and $\baire$, such that each one is listed cofinally. Construct a sequence $\{ M_\alpha : \alpha \in \mathfrak{c} \}$ of subelementary models of some $H(\lambda)$ such that for all $\alpha < \mathfrak{c}$
    \begin{itemize}
        \item $|M_\alpha| < \mathfrak{c}$,
        \item $(\bigcup_{\beta < \alpha} M_\beta) \cup \{ A_\alpha, f_\alpha, K_\alpha, c_\alpha \} \subseteq M_\alpha$,
    \end{itemize}
    where $\{ c_\alpha : \alpha \in \mathfrak{c} \}$ is a collection of functions such that, for all $\alpha \in \mathfrak{c}$,  $c_{\alpha + 1} \in K_\alpha$ and $c_{\alpha + 1}$ avoids all meager sets of $K_\alpha$ that live in $M_\alpha$.
    Clearly this is possible under $\cov{\mathcal{M}}=\mathfrak{c}$. We will construct recursively an increasing sequence  $\{ F_\alpha : \alpha \in \mathfrak{c}\}$ such that $F_0 \subseteq M_0$ and for each $\alpha < \mathfrak{c}$,
    \begin{enumerate}
        \item $F_\alpha$ is a filter basis and $F_\alpha \subseteq M_{\alpha + 1}$,
        \item if $K_\alpha$ is internally unbounded in the filter generated by $\bigcup_{\beta < \alpha}F_\beta$ or if $K_\alpha = \baire$, then, for every $f \in K_\alpha \cap M_\alpha$ the set $\{ n \in \omega : f(n) \leq c_{\alpha + 1}(n) \} \in F_\alpha$.
    \end{enumerate}
    Assume that $F_\beta$ has been constructed with these properties for $\beta < \alpha$ and that $K_\alpha$ is internally unbounded in $\bigcup_{\beta < \alpha}F_\beta$ (the case $K_\alpha = \baire$ is similar). For each $A \subseteq \omega$ and each $f \in \baire$ let $B(A,f) = \{ n \in A : f(n) \leq c_{\alpha + 1}(n) \}$. Since $K_\alpha$ is internally unbounded in $\bigcup_{\beta < \alpha}F_\beta$, then $B(P,f)$ is infinite for each $P \in (\bigcup_{\beta < \alpha} F_\beta)^+$ and each $f \in K_\alpha \cap M_\alpha$ (see the Remark above). We will show that $\bigcup_{\beta < \alpha}F_\beta \cup \{ B(F,f) : f \in K_\alpha \cap M_\alpha, F \in \bigcup_{\beta < \alpha} F_\beta \}$ generates a filter basis: It is enough to show that, if $f_0, f_1, \ldots, f_k \in K_\alpha \cap M_{\alpha}$ and $F \in \bigcup_{\beta < \alpha} F_B$, then $\bigcup_{i \leq k} B(F,f_i)$ is infinite: For each $i \leq k$ let $D_i = \{j \in F : f_i(j) = \max \{f_m (j) : m \leq k\}\}$. Clearly all of the $D_i$ are in $M_\alpha$ and one of the $D_i$ is positive according to the filter generated by $\bigcup_{\beta < \alpha} F_\beta$, so for that $i$, $B(D_i,f_i)$ is infinite and is clearly contained in all of the $B(F_j,f_k).$ Then we let $F_\alpha$ be the collection of all possible finite intersections of $\bigcup_{\beta < \alpha}F_\beta \cup \{ B(F,f) : f \in K_\alpha \cap M_\alpha, F \in \bigcup_{\beta < \alpha} F_\beta \}$. It follows that $F_\alpha$ satisfies the properties we require.

    Let $\mathscr{U}$ be any ultrafilter extending $\bigcup_{\beta < \mathfrak{c}} F_\beta$, we will show that $\mathscr{U}$ is a Michael ultrafilter: First notice that if $\gamma : \cof{\mathfrak{c}} \rightarrow \mathfrak{c}$ is a cofinal sequence such that, for each $\alpha \in \cof{\mathfrak{c}}$, $K_{\gamma(\alpha)} = \baire$, then $\mathscr{D} = \{c_{\gamma(\alpha) + 1} : \alpha \in \cof{\mathfrak{c}}\}$ is a $\leq_\mathscr{U}$-dominating family of minimal size: First, for every $f \in \baire$, there is a ordinal of the form $\gamma(\beta)$ such that $f \in M_{\gamma(\beta)}$, and therefore (2) implies that $f \leq_\ult{U} c_{\gamma(\beta) + 1}$, so $\mathscr{D}$ is $\leq_\ult{U}$-dominating. On the other hand, if $\mathcal{F} \subseteq \baire$ is of size $< \cof{\mathfrak{c}}$, then there is a ordinal of the form $\gamma(\beta)$ such that $\mathcal{F} \subseteq M_{\gamma(\beta)}$. By similar reasons, $\mathcal{F}$ does not bound $c_{\gamma(\beta) + 1}$ so $\mathscr{D}$ is of minimal size, thus $\cof{\ultpow{\baire}} = \cof{\mathfrak{c}}$.

    If $K \subseteq \baire$ is a compact set such that $\cof{\ultpow{K}} > \omega$, then there is a compact $K'\subseteq K$ internally unbounded in $\mathscr{U}$ such that $\cof{\ultpow{K}} = \cof{\ultpow{K'}}$ so $K'$ is internally unbounded in the filter generated by $F_\alpha$ for every $\alpha < \mathfrak{c}$. By a similar argument as above, it follows that $\cof{\ultpow{K'}} = \cof{\mathfrak{c}}$. As a consequence, $\mathscr{U}$ is a Michael ultrafilter.
\end{proof}

A class of ultrafilters $\mathcal{C}$ \emph{exists generically} if every filter basis of size smaller than $\mathfrak{c}$ can be extended to an ultrafilter of the class $\mathcal{U}$. The generic existence of p-points is equivalent to the fact that $\mathfrak{d} = \mathfrak{c}$ (a proof can be found in \cite{ketonenppoints} or \cite{Barty}) and the existence of selective ultrafilters is equivalent to $\cov{\mathcal{M}}= \mathfrak{c}$. This notion has been studied and expanded in \cite{BrendleGenericExistence}. The previous theorem provides a natural question.

\begin{question}
    In which situations a small basis for a filter can be extended to a Michael ultrafilter?
\end{question}

Under $\mathfrak{d} = \aleph_1$ every countable basis for a filter can be extended to a Michael ultrafilter and later we will show that, under $\mathfrak{g} < \mathfrak{u}$ this would be impossible, no matter what notion of smallness we consider.

\section{Selective ultrafilters are almost Michael.}

In this section we will focus on proving that selective ultrafilters are Michael, provided that $\cof{\ultpow{\baire}} \leq \cin{\ultpow{\baire}}$. First, we will start with an example of a compact set that has an uncountable internal $\leq_\ult{U}$-cofinality.

Given a set $A \subseteq \omega$ such that $A= \{a_i : i \in \omega\}$, we define $\varphi_A$ as follows:

$$
\varphi_A (n) = \begin{cases}
    a_0 & n = a_0 \\
    a_{k+1} - a_k & n = a_{k+1}\\
    0 & \text{otherwise}.
\end{cases}
$$

Let $K_0 = \{ \varphi_A : A \subseteq \omega \}$. The relation $A \rightarrow \varphi_A$ is a continuous function, so $K_0$ is a compact subspace of $\baire$. If $i \in A\cap B$, then $\varphi_A (i) \geq \varphi_B (i)$ means that the previous number in $A$ below $i$ is smaller than the previous number of $B$ below $i$. The following proposition will provide bounds for $\cof{\ultpow{K_0}}$ whenever $\ult{U}$ is a p-point.

\begin{proposition} Let $\mathscr{U}$ be an ultrafilter on $\omega$. Then
\begin{itemize}
    \item $\cof{\ultpow{K_0}} \leq \chi (\mathscr{U})$,
    \item If $\mathscr{U}$ is a p-point, then $\aleph_0 < \cof{\ultpow{K_0}} $.
\end{itemize} \label{boundsforppoints}
\end{proposition}
\begin{proof} Both points easily follow from the fact that if $A \subseteq^*B$ then $\varphi_A | A \geq^* \varphi_B | A$.
\end{proof}

We will be able to be more precise when we consider q-points. For that purpose, we will need to study the coinitiality of the ultrapower, originally considered by M. Canjar in \cite{Canjar2}. The author proved that $\mathfrak{b} \leq \cin{\ultpow{\baire}} \leq \mathfrak{d}$. Alternatively, the reader can convince themself that, when restricted to the space of finite to one functions from $\omega$ to $\omega$, $\mathfrak{d}$ is the cardinality of the smallest family that is dominating with the reverse order $\geq^*$ and $\mathfrak{b}$ is the cardinality of the smallest unbounded family with the reverse order and therefore $\mathfrak{b} \leq \cin{\ultpow{\baire}} \leq \mathfrak{d}$. In \cite{BlassHeike}, the authors proved that $\mathfrak{g} \leq \cof{\ultpow{\baire}}$. Using the same ideas, we were able to get a similar bound for the coinitiality.

\begin{theorem}
    For every ultrafilter $\mathscr{U}$, $\mathfrak{g} \leq \cin{\ultpow{\baire}}$. \label{blassheikeforcoinitiality}
\end{theorem}
\begin{proof}
Let $\{ d_\alpha : \alpha < \cin{\ultpow{\baire}} \}$ be a $\leq_\mathscr{U}$-coinitial family of finite to one functions. Given $X \subseteq \omega$ infinite, consider a function $\nu_X$ defined by $\nu_X (n) = |X \cap n|$. Clearly $\nu_X$ is non-decreasing and finite to one.

    For each $\alpha < \cin{\ultpow{\baire}}$ let
    $$
    G_\alpha = \{ X \in [\omega]^\omega : \forall k \in \omega \ ( d_\alpha \geq_\mathscr{U} \nu_X + k )\}.
    $$

    Since the family of the $d_\alpha$ is $\leq_\mathscr{U}$-coinitial, we have that $\bigcap_{\alpha < \mathrm{cin}(\baire / \mathscr{U})}G_\alpha = \emptyset$. We only have to show that each $G_\alpha$ is groupwise dense:

If $X \subseteq^* Y$, then there is $k\in \omega$ such that $Y \cap n$ has more elements than $X \cap n \setminus k$ for all $n \in \omega$. Thus $\nu_X \leq \nu_Y + k $ for some $k \in \omega$. So it follows that $G_\alpha$ is closed under subsets and finite modifications. For the other part let $\{ I_n : n \in \omega \}$ be an interval partition of $\omega$. Inductively, pick a subfamily $\{ I_{n_k} : k\in \omega \}$ such that if $\ell = \min I_{n_{k + 1}}$ then $d_\alpha(\ell) > |\bigcup_{j \leq k} I_{n_j}  | + k$. Let $X = \bigcup_{i \in \omega} I_{n_{2i}}$ and $Y = \bigcup_{i \in \omega} I_{n_{2i + 1}}$. To finish the proof, we will observe that, given an $i\in\omega$, $\{ n\in \omega : \nu_X (n) + i < d_\alpha (n) \} \cup \{ n\in \omega : \nu_Y (n) + i < d_\alpha (n) \}$ is cofinite:
Let $\ell > \min I_{n_{i+1}}$ and assume that $j$ is the smallest number such that $\max I_{n_j} < \min I_{n_{j + 1}} \leq \ell$. We will assume that $j = 2i$ for some $i \in \omega$ (the other case is similar). Then $\nu_X (\ell) = \nu_X (\max I_{2j} + 1) \leq \nu_{X \cup Y} (\min I_{2j + 1}) < d_\alpha (\min I_{2j + 1}) - i \leq d_\alpha (\ell) - i$. In other words $\ell \in \{ n\in \omega : \nu_X (n) + i < d_\alpha (n) \}$. This implies that either $\nu_X + i \leq_{\mathscr{U}} d_\alpha$ or $\nu_Y + i \leq_{\mathscr{U}} d_\alpha$, so taking the one that repeats for infinitely many $i$ we get that either $X \in G_\alpha$ or $Y \in G_\alpha$.

    Therefore $\{ G_\alpha : \alpha < \cin{\ultpow{\baire}} \}$ is a family of groupwise dense sets with empty intersection, thus $\mathfrak{g} \leq \cin{\ultpow{\baire}}$.
\end{proof}

We will now give a characterization of $\cin{\ultpow{\baire}}$ for whenever $\ult{U}$ is a q-point. For every $A \in \Baire$, pick any non-decreasing $h_A \in \baire$ such that for all $i \in A$, $i$ is the $h_A(i)$-th element of $A$. Clearly $h_A$ is finite to one.

We have the following easy observation:

\begin{remark}
If $\mathscr{U}$ is a q-point, then for every finite to one $f\in \baire$ there is $A \in \mathscr{U}$ such that $h_A \leq f.$
\end{remark}
\begin{proof}
    Without loss of generality, $f$ is non-decreasing. Pick $A \in \mathscr{U}$ such that for all $n\in\omega$, $|A\cap f^{-1}(n)| \leq 1$. Clearly $h_A \leq f$.
\end{proof}

Before we continue, we would like to recall a well-known property of q-points.

\begin{lemma}
    If $\ult{U}$ is a q-point and $A \subseteq \omega$ is infinite, then there is $U = \{ u_i : i \in \omega \} \in \ult{U}$ such that for all $i \in \omega$, $|[u_i, u_{i + 1}) \cap A| \geq u_i + 1$.
\end{lemma}
\begin{proof}
    If $A$ is infinite, let $A'\subseteq A$ be such that, if $A'=\{a_i : i \in \omega \}$, then $|[a_i,a_{i + 1}) \cap A| = a_i + 1$. If $U = \{ u_j : j \in \omega \} \in \ult{U}$ is such that, for every $j$ there is $i$ such that $[a_i, a_{i+1}) \subseteq [ u_j, u_{j+1} )$, then from this and from the fact that $u_i \leq a_i$, it follows that for all $j \in \omega$, $|[u_j, u_{j + 1}) \cap A| \geq u_j + 1$
\end{proof}

Now we are ready to calculate the cardinal $\cof{\ultpow{K_0}}$ whenever $\ult{U}$ is a q-point.

\begin{proposition}
    If $\mathscr{U}$ is a q-point, then $\cof{\ultpow{K_0}} = \cin{\ultpow{\baire}}$. \label{boundsforqpoints}
\end{proposition}
\begin{proof} ($\cof{\ultpow{K_0}} \geq \cin{\ultpow{\baire}}$). Let $\{ \varphi_{A_\alpha} : \alpha < \cof{\ultpow{K_0}} \}$ be $\leq_\mathscr{U}$-cofinal in $K_0$, we will show that $h_{A_\alpha}$ is $\leq_{\mathscr{U}}$-coinitial: Let $f$ be a finite to one non-decreasing function. By using the previous remark we may assume that $f = h_A$ for some $A \in \mathscr{U}$. Find $U = \{ u_i : i \in \omega \} \in \mathscr{U}$ such that, for all $i \in \omega, |[u_i , u_{i+1}) \cap A| \geq u_i + 1$. Since
$\{ \varphi_{A_\alpha} : \alpha \in \cof{\ultpow{K_0}} \}$ is $\leq_\mathscr{U}$-cofinal in $K_0$, there must be an $\alpha$ such that $\varphi_U \leq_V \varphi_{A_\alpha}$ for some $V \in \mathscr{U}$. We will show that $f_{A_\alpha} \leq_{A\cap U \cap V \cap A_\alpha} f_A$: Let $i \in A\cap U \cap V \cap A_\alpha$, $k$ be the previous element of $U$ below $i$ and $j$ the previous element of $A_\alpha$ below $i$. First note that $f_{A_\alpha} (i) \leq j + 1$, since $i$ is the next element of $A_\alpha$ above $j$. Note that $[k,i)\cap A$ has at least $k + 1$ elements, so $f_A(i) \geq k + 1$. Since $i \in V \cap A_\alpha \cap U$ then $0 < \varphi_U(i) \leq \varphi_{A_\alpha}(i)$, meaning that $j \leq k$, so $f_{A_\alpha} (i) \leq j + 1 \leq k + 1 \leq f_A$.

($\cof{\ultpow{K_0}} \leq \cin{\ultpow{\baire}}$). Let $\lambda < \cof{\ultpow{K_0}}$ and let $\{ h_{A_\alpha} : \alpha < \lambda \} \subseteq P$. We will find $U \in \mathscr{U}$ such that $h_U \leq_\mathscr{U} h_{A_\alpha}$ for every $\alpha \in \lambda$: For every $\alpha$ pick $U_\alpha \in \mathscr{U}$ such that, if $U_\alpha = \{u^\alpha_i: i\in \omega\}$ then for all $i \in \omega, |[u_i^\alpha, u_{i+1}^\alpha) \cap A_\alpha| \geq u_i^\alpha + 1$. Find $U \in \mathscr{U}$ such that $\varphi_U \geq_{V_\alpha} \varphi_{U_\alpha}$ for some $V_\alpha \in \mathscr{U}$. We will show that $h_U \leq_{U\cap U_\alpha \cap V_\alpha \cap A_\alpha} h_\alpha$: Let $i \in U\cap U_\alpha \cap V_\alpha \cap A_\alpha$, $k$ be the previous element of $U_\alpha$ below $i$ and $j$ be the previous element of $U$ below $i$. Clearly $h_U(i) \leq j + 1$ and, since $|[k, i) \cap A_\alpha| \geq k + 1$ then $h_{A_\alpha} (i) \geq k + 1$. To finish the proof, note that since $i \in U \cap V_\alpha \cap U_\alpha$ then $0 < \varphi_{U_{\alpha}}(i) \leq \varphi_{U}(i)$, so $j \leq k$ so $h_{U} (i) \leq j + 1 \leq k + 1 \leq h_{A_\alpha}(i)$.
\end{proof}

Later we will show that the previous proposition may fail if we assume that $\ult{U}$ is a p-point instead of a q-point. We now know that $\cof{\ultpow{K_0}}$ is uncountable whenever $\ult{U}$ is either a p-point or a q-point. In general, we do not know the answer of the following:

\begin{question}
    Is there an ultrafilter $\ult{U}$ such that $\cof{\ultpow{K_0}} = \aleph_0$?
\end{question}

We know that there is always an ultrafilter $\ult{U}$ and a compact set $K$ such that $\cof{\ultpow{K}}$ is uncountable: Let $\mathscr{F}$ be any $\sigma$-compact p-filter (for example, the dual filter of the sumable ideal $\{ A \subseteq \omega : \sum_{i \in A} \frac{1}{i} < \infty \}$). Consider the set $K_\mathscr{F} = \{ \varphi_F : F \in \mathscr{F} \}$ (we refer the reader to the definition of $K_0$ at the beginning of this section) and let $\mathscr{U}$ be any  ultrafilter extending $\mathscr{F}$. The proof of Proposition \ref{boundsforppoints} shows that $\cof{\ultpow{K_\mathscr{F}}}$ must be uncountable, and since $K_\mathscr{F}$ is $\sigma$-compact, then there must be a compact set $K$ such that $\cof{\ultpow{K_\mathscr{F}}} = \cof{\ultpow{K}}$. We do not know if one can always construct such a compact set for an arbitrary ultrafilter.

\begin{question}
    Is it always possible, for every ultrafilter $\ult{U}$, to construct a compact set $K \subseteq \baire$ such that $\cof{\ultpow{K}}$ is uncountable? \label{question3}
\end{question}

This question will be partially answered in the final section of the paper, where we will prove that, under some conditions, there are no Michael ultrafilters. The reader interested in the theory of ideals and filters on countable sets may consult \cite{MichaelSurvey}.

We will now focus on selective ultrafilters. We will need to use an additional collection of functions. Given a compact set $K \subseteq \baire$ and $A \subseteq \omega$ we define $f_A^K \in K$ recursively: Assume $A = \{a_i : i \in \omega\}$, then:
\begin{itemize}
    \item $f_A^K | (a_0 + 1)$ is such that $f_A^K (a_0)$ has a maximum value among all $g \in K$,
    \item $f_A^K | (a_{i+1} + 1)$ is such that $f_A^K (a_{i+1})$ has a maximum value among all $g \in K \cap \langle f_A^K | a_i + 1 \rangle$.
\end{itemize}

For our next result, we will use the following well-known characterization of a selective ultrafilter, whose proof can be found in \cite{SergeUtrees}:
\begin{theorem}
    An ultrafilter $\ult{U}$ is selective if and only if every $\ult{U}$-branching tree $T$ (i.e. $T$ is a tree in $\bairetree$ and every node of $T$ splits into an element of the ultrafilter $\ult{U}$) $T$ has a branch $b \in [T]$ such that $\{b(n) : n \in \omega \} \in \ult{U}$.
\end{theorem}

We are ready to prove the following.

\begin{proposition}
    Assume $\mathscr{U}$ is a selective ultrafilter, if $K = [T] \subseteq \baire$ is a compact set internally unbounded on $\mathscr{U}$ then $\cof{\ultpow{K}} \geq \cof{\ultpow{K_0}}$. \label{K0selective}
\end{proposition}
\begin{proof}
    Let $\mathcal{F} = \langle f_\alpha : \alpha \in \lambda \rangle \subseteq K$ with $\lambda < \cof{\ultpow{K_0}}$, we will show that $\mathcal{F}$ is $\leq_{\mathscr{U}}$-bounded in $K$. For every $\alpha < \lambda$ and $s \in T$ pick $g_s^\alpha \in \cone{s} \cap K$ and $U_s^\alpha$ such that $g_s^\alpha \geq_{U_s^\alpha} f_\alpha$, which is possible since $K$ is internally unbounded in $\mathscr{U}$. The $U_s^\alpha$ may be constructed in such a way that
    \begin{itemize}
        \item if $s \subseteq t$ then $U_t^\alpha \subseteq U_s^\alpha$,
        \item $U_s^\alpha \cap (|s| + 1) = \emptyset$,
        \item if $|s| = |t| = m$ then $U_s^\alpha = U_t^\alpha = U_m^\alpha$.
    \end{itemize}
    Let $T_\alpha$ be a tree constructed in the following way:
    \begin{itemize}
        \item $\emptyset \in T_\alpha$,
        \item if $s \in T_\alpha$ then $\mathrm{succ}_{T_\alpha}(s) = U_m^\alpha$, where $m = s(|s|)$.
    \end{itemize}
    Then $T_\alpha$ is a $\mathscr{U}$-branching tree, so there is $b_\alpha \in [T_\alpha]$ such that $U_\alpha = \{b(n) : n \in \omega \} \in \mathscr{U}$.
    The following is a crucial property of $U_\alpha$.

\begin{remark}
    For every $s \in T$ such that $|s| \in U_\alpha$ and every $i \in U_\alpha$ with $|s| < i$ there is $g^\alpha_s \in K$ such that $s \subseteq g^\alpha_s$ and $g^\alpha_s(i) \geq f_\alpha(i)$.
 \end{remark}  
 \begin{proof}[Proof of the remark.] 
 Note that if $m = |s|$, then all the elements of $U_\alpha$ after $m$ have to be picked from $U^\alpha_m$, so therefore it follows that $i \in U^\alpha_{|s|} = U^\alpha_{s}$, thus $f_\alpha (i) \leq g_s^\alpha (i)$. 
 \end{proof}

The set $\{ \varphi_{U_\alpha}: \alpha \in \lambda \} \subseteq K_0$ is not $\leq_\mathscr{U}$-dominating, so there must be an $U = \{ u_i : i \in \omega \} \in \Baire$ such that for all $\alpha \in \lambda$, $\varphi_{U_\alpha} \leq_{V_\alpha} \varphi_U$ for some $V_\alpha \in \mathscr{U}$.

We will show that $f_\alpha \leq_{U\cap U_\alpha \cap V_\alpha} f_U^K$ for every $\alpha \in \lambda$: Let $i \in U\cap U_\alpha \cap V_\alpha$. Since $i \in U$, then $i = u_\ell$ for some $\ell \in \omega$. We may assume that $\ell = k+1$ (since it is trivially true for $\ell = 0$). Since $i \in U\cap U_\alpha \cap V_\alpha$, then $0 < \varphi_{U_\alpha} (i) \leq \varphi_U(i)$ so there must be an $m \in U_\alpha$ such that $u_k \leq m < u_{k+1} = i$ and, by the remark above, for $s = f^K_U | m$, there is a $g_s \in K$ such that $s \subseteq g_s$ and $f_\alpha (u_{k+1}) \leq g^\alpha_s (u_{k+1})$. Therefore, by the definition of $f^K_U(u_{k+1})$, $f_\alpha (u_{k+1}) \leq g^\alpha_s (u_{k+1}) \leq f^K_U(u_{k+1})$ and the proof is complete.
\end{proof}

Finally we are ready to prove the main theorem of this work.

\begin{theorem}
    A selective ultrafilter $\ult{U}$ is Michael if and only if $ \cof{\ultpow{\baire}}\leq \cin{\ultpow{\baire}}$. \label{maintheorem}
\end{theorem}
\begin{proof}
It follows from Propositions \ref{boundsforqpoints}, and \ref{K0selective}.
\end{proof}

As a consequence, there is a Michael space whenever there is a selective ultrafilter and $\max \{ \mathfrak{b}, \mathfrak{g} \} = \mathfrak{d}$ (as this implies that $\cof{\ultpow{\baire}} = \cin{\ultpow{\baire}}$ for all ultrafilters $\ult{U}$). It is important to notice that this is not a characterization. For example, in the Cohen model, there are selective ultrafilters such that $\cof{\ultpow{\baire}} = \cin{\ultpow{\baire}}$ (see \cite{Canjar2}, or notice that the proof of Theorem \ref{michaelcovmeager} can be modified so the resulting ultrafilter is Michael and selective), but $\max\{ \mathfrak{b}, \mathfrak{g}\} = \aleph_1 < \mathfrak{d} = \cov{\mathcal{M}} = \mathfrak{c}$. Another example can be obtained by forcing with $\mathcal{P}(\omega) / \mathrm{Fin}$; the generic filter is a selective ultrafilter and in the extension $\mathfrak{b} = \mathfrak{c}$ (see \cite{Barty}), therefore, the generic filter is a Michael ultrafilter. In \cite{Canjar2} M. Canjar showed that, in the Cohen model, there are ultrafilters such that $\cof{\ultpow{\baire}} > \cin{\ultpow{\baire}}$, although it is not clear that they can be constructed to be selective, arising the following question.

\begin{question}
    Is it true that, if $\ult{U}$ is selective, then $\cof{\ultpow{\baire}} = \cin{\ultpow{\baire}}$?
\end{question}

In the next section, we will show that there is a model with no Michael ultrafilters. In this model there are p-points but no q-points. We will this section with the following question.

\begin{question}
    Is it possible to generalize Theorem \ref{maintheorem} to q-points? \label{question5}
\end{question}

\section{A model without Michael ultrafilters}

The goal of this section is to prove that the inequality $\mathfrak{u}< \mathfrak{g}$ implies that no ultrafilter can be Michael. We will start this section showing that, under this hypothesis, some p-points cannot be Michael.

\begin{proposition}
    Under $\mathfrak{u}< \mathfrak{g}$, if $\ult{U}$ is a p-point of character $\mathfrak{u}$, then $\ult{U}$ cannot be Michael. \label{RBppointofcharacberu}
\end{proposition}
\begin{proof}
    On one hand, by Blass' and Mildenberger's theorem (see \cite{BlassHeike}), we have that $\cof{\ultpow{\baire}} \geq \mathfrak{g}$. On the other, Proposition \ref{boundsforppoints} implies that $\aleph_0 < \cof{\ultpow{K_0}} \leq  \mathfrak{u}$.
\end{proof}

We will recall the Rudin-Keisler and the Rudin-Blass orders for ultrafilter as some of the results of this section can be easily stated in that language. 

\begin{definition}
    Given two ultrafilters $\mathscr{U},\mathscr{V}$ on $\omega$, $\mathscr{U}$ is \emph{Rudin-Keisler below} $\mathscr{V}$ (denoted by $\mathscr{U} \leq_{RK} \mathscr{V}$) if there is a function $f: \omega \rightarrow \omega$ (which we will call the \emph{witness function}) such that $U \in \mathscr{U}$ if and only if $f ^{-1}[U] \in \mathscr{V}$. If, additionaly, such function can be found finite to one, then $\mathscr{U}$ is \emph{Rudin-Blass below} $\mathscr{V}$ (denoted by $\mathscr{U} \leq_{RB} \mathscr{V}$).
\end{definition}

These orders have been studied extensively. The reader interested in these orders can consult \cite{MichaelSurvey}, \cite{Nyikosultrafilter}, and \cite{LaflammeRB}. Two basic properties that we will need are the following, which are very easy to prove.

\begin{proposition} Assume that $ \mathscr{U} \leq_{RB} \mathscr{V}$: \label{RBbasic}
\begin{itemize}
    \item If $\mathscr{V}$ is a p-point, then $\mathscr{U}$ is a p-point,
    \item $\chi(\mathscr{V}) \geq \chi(\mathscr{U})$.
\end{itemize}
\end{proposition}

One important fact about models of $\mathfrak{u} < \mathfrak{g}$ is that p-points of character $\mathfrak{u}$ always exist, as a consequence of a well-known theorem by Ketonen (\cite{ketonenppoints}) and the fact that $\mathfrak{g} \leq \mathfrak{d}$.

\begin{theorem}
    If $\ult{U}$ is an ultrafilter such that $\chi(\ult{U}) < \mathfrak{d}$, then  $\ult{U}$ is a p-point.
\end{theorem}

Our main focus now will be to extend Proposition \ref{RBppointofcharacberu} to all ultrafilters. First, we will need the following simple lemma.
\begin{lemma}
    Assume that $\mathscr{U} \leq_{RK} \mathscr{V}$, $f$ is the witness function and $h_0,h_1 \in \baire$. Then $h_0 \leq_\mathscr{U} h_1$ if and only if $h_0 \cdot f \leq_\mathscr{V} h_1 \cdot f$.
\end{lemma}
\begin{proof}
    It follows from the definitions.
\end{proof}

We will now try to analyze all possible $\leq_\ult{U}$ cofinalities of compact sets. 

\begin{definition}
    Given an ultrafilter $\mathscr{U}$, \emph{the spectrum of compact sets} is
$$
\mathsf{spec}(\mathscr{U}) = \{ \cof{\ultpow{K}} : K\subseteq \baire \text{ is compact}\}.
$$
\end{definition}

We can easily show that the spectrum will get smaller whenever we consider ultrafilters that are lower in the Rudin-Keisler order.

\begin{proposition}
    If $\mathscr{U} \leq_{RK} \mathscr{V}$, then $\mathsf{spec}(\mathscr{U}) \subseteq \mathsf{spec}(\mathscr{V})$. \label{RBspectrumgetssmaller}
\end{proposition}
\begin{proof}
    Let $f: \omega \rightarrow \omega$ be the witness function of $\mathscr{U} \leq_{RK} \mathscr{V}$ and let $K \subseteq \baire$ be a compact set. Then $K' = \{ g \cdot f : g\in K \}$ is compact. The conclusion follows easily from the lemma above.
\end{proof}

In general, we do not know anything else about the spectrum of compact sets, even when considering the Rudin-Blass order. We do know that the cofinalities of the ultrapowers stays the same when considering two compatible ultrafilters in the Rudin-Blass order. The following is also proven in \cite{Nyikosultrafilter}.

\begin{proposition}
    If $\mathscr{U} \leq_{RB} \mathscr{V}$ then $\cof{\ultpow{\baire}} = \cof{\vltpow{\baire}}$. \label{RBequalcofinalities}
\end{proposition}
\begin{proof}
     Assume that $f: \omega \rightarrow \omega$ is the finite to one witness of $\mathscr{U} \leq_{RB} \mathscr{V}$. 
     
     ($\leq$) Let $\{ g_\alpha : \alpha < \cof{\vltpow{\baire}} \}$ be a $\leq_\mathscr{V}$-dominating family. We may additionally assume that each $g_\alpha$ is constant on each $f^{-1}(n)$. We will show that $\{ \Hat{g}_\alpha : \alpha \in \cof{\vltpow{\baire}} \}$ is $\leq_\mathscr{U}$-dominating, where $\Hat{g}_\alpha$ is such that $\Hat{g}_\alpha \cdot f = g_\alpha$ (which is possible to find since $g_\alpha$ is constant on $f^{-1}(k)$): Let $h \in \baire$. Construct $h'$ such that $h'$ is constant on each $f^{-1}(n)$ and $h' = h \cdot f$. Pick $\alpha$ such that $g_\alpha = \Hat{g}_\alpha \cdot f \geq_\mathscr{V} h \cdot f$. By the lemma above $\Hat{g}_\alpha \geq_\mathscr{U} h$.

    ($\geq$) On the other hand, if $\{ h_\alpha : \alpha < \cof{\ultpow{\baire}} \}$ is $\leq_\mathscr{U}$-dominating, then $\{ h_\alpha \cdot f : \alpha < \cof{\ultpow{\baire}} \}$ is $\leq_\mathscr{V}$-dominating: Let $g \in \baire$. We may assume that $g$ is constant on each $f^{-1}(n)$, so there is a $\Hat{g}$ such that $\Hat{g} \cdot f = g$. Pick $\alpha$ such that $\Hat{g} \leq_\mathscr{U} h_\alpha$, then, by the lemma, $g = \Hat{g} \cdot f \leq_\mathscr{V} h_\alpha \cdot f$.
\end{proof}

A similar argument can be used to show the analogous proposition for the coinitialities of the ultrapowers. These propositions can be used to prove the following

\begin{corollary}
    The property of being Michael is closed downwards in the Rudin-Blass order, ie if $\mathscr{V}$ is Michael and $\mathscr{U} \leq_{RB }\mathscr{V}$ then $\mathscr{U}$ is Michael.
\end{corollary}
\begin{proof}
This is an immediate consequence of Propositions \ref{RBspectrumgetssmaller} and \ref{RBequalcofinalities}.
\end{proof}

We do not know if the last corollary can be generalized it to the Rudin-Keisler order. An important application of $\mathfrak{u} < \mathfrak{g}$ is the principle of near coherence of filters, introduced in \cite{BlassNCF}. The \emph{near coherence of filters} state that the Rudin-Blass order is downwards directed, that is, that every two ultrafilters have a lower bound in the Rudin-Blass order. In \cite{BlassLaflamme}, A. Blass and C. Laflamme show that the near coherence of filters is a consequence of $\mathfrak{u} < \mathfrak{g}$. We will now prove the main theorem of this section.

\begin{theorem}
    Under $\mathfrak{u} < \mathfrak{g}$ no ultrafilter is Michael.
\end{theorem}
\begin{proof}
Let $\ult{U}$ be an ultrafilter and let $\ult{V}$ be a p-point of character $\mathfrak{u}$. Then, by the near coherence of filters, there is an ultrafilter $\ult{W}$ such that $\ult{W} \leq_{RB} \ult{U},\ult{V}$. By Propositions \ref{RBppointofcharacberu} and \ref{RBbasic}, $\ult{W}$ is not a Michael ultrafilter, and since being Michael is downwards closed, $\ult{U}$ cannot be Michael either.
\end{proof}

No q-point can exist under the presence of Near Coherence of filters, since being q-point is closed downwards in the Rudin-Blass ordering and there is always an ultrafilter with no q-point below in the Rudin-Blass ordering (see \cite{LaflammeRB}), giving no partial answer to Question \ref{question5}. Finally, it is worth to mention that the inequality $\mathfrak{u} < \mathfrak{g}$ is consistent. A model that satisfies $\mathfrak{u}<\mathfrak{g}$ can be obtained by forcing with a countable support iteration of length $\omega_2$ of Miller's forcing (see \cite{BlassSurvey}). In this model, $\mathfrak{b} = \omega_1$, so there is a Michael space in this model. 

To finish this work, we would like to mention that the existence of Michael ultrafilters can be decided in the most common models of set theory. In the Cohen model, there are Michael ultrafilters since $\cov{\mathcal{M}} = \mathfrak{c}$, so does in every model obtained by forcing with a long finite support iteration of c.c.c. forcings over a model of CH. All ultrafilters are Michael in the Random, Silver and Sacks model since $\mathfrak{d} = \aleph_1$ (we refer the reader to \cite{Barty} for the definition and properties of these models). There are no Michael ultrafilters in the Miller's model, and there is a Michael ultrafilter after forcing with $\mathcal{P}(\omega) / \mathrm{Fin}$. However there is little we know about Michael ultrafilters in the Mathias or in the Laver model; models obtained by forcing with a countable support iteration of Mathias' or Laver's forcing respectively over a model of CH. We conclude this work with the following question.

\begin{question}
    Is there a non-Michael ultrafilter in either the Laver or in the Mathias model?
\end{question}

\textbf{Acknowledgements:} We would like to thank Michael Hru\v{s}\'{a}k and the Set-Theory and Topology seminar groups from the Instytut Matematyczny, Uniwersytet Wrocławski and from the Posgrado Conjunto de Ciencias Matem\'{a}ticas, UNAM-UMSNH for many hours of stimulating conversations.

\bibliographystyle{alphadin} 
\bibliography{main}
\end{document}